\documentclass{article}
\usepackage[utf8]{inputenc}
\usepackage[a4paper]{geometry}
\RequirePackage[l2tabu, orthodox]{nag}
\usepackage{microtype}
\usepackage[colorlinks]{hyperref}
\usepackage{times}

\usepackage{amstext}
\usepackage{amsmath}
\usepackage{amssymb}
\usepackage{amsthm}
\usepackage{amsrefs}

\usepackage{enumerate}
\usepackage{xspace}
\usepackage{mathtools}
\usepackage{stmaryrd}

\usepackage[colorinlistoftodos]{todonotes}

\usepackage{tikz}
\usetikzlibrary{external}
\usetikzlibrary{calc}

\usepackage{subfig}
\usepackage[heightadjust=all,valign=c]{floatrow}
\usepackage{fr-subfig}

\theoremstyle{plain}

\newtheorem{Thm}{Theorem}

\newtheorem{Pro}[Thm]{Proposition}

\newcommand{\WW}{\mathcal{W}}

\DeclareMathOperator{\Prop}{\mathrm{PROP}}

\newcommand{\LEM}{\textrm{LEM}\xspace}
\newcommand{\WLEM}{\textrm{WLEM}\xspace}
\newcommand{\DNE}{\textrm{DNE}\xspace}
\newcommand{\DGP}{\textrm{DGP}\xspace}
\newcommand{\DGPimp}{\ensuremath{\textrm{DGP}^\implies}\xspace}

\newcommand{\EFQ}{\textrm{EFQ}\xspace}
\newcommand{\WT}{\textrm{WT}\xspace}

\newcommand{\PP}{\textrm{PP}\xspace}

\newcommand{\KP}{\textrm{KP}\xspace}
\newcommand{\SmL}{\textrm{SmL}\xspace}

\renewcommand{\mid}{\, \middle| \,}
\newcommand{\menge}[1]{\ensuremath{\left\{ #1 \right\}}}
\newcommand{\set}[2]{\mbox{$ \left\{ \,#1 \mid #2 \,\right\}$}}
\newcommand{\Pow}[1]{\mathcal{P}\left(#1\right)} 
\newcommand{\fa}[2]{\forall {#1} : {#2}}

\newcommand{\define}[1]{\textit{#1}}

\usepackage{pifont}
\newcommand{\cmark}{\ding{51}}
\newcommand{\xmark}{\ding{55}}
\newcommand{\ext}[1]{\llbracket #1  \rrbracket}

\renewcommand{\implies}{\rightarrow}
\renewcommand{\iff}{\leftrightarrow}
\newcommand{\mimplies}{\mathchoice{\Longrightarrow}{\Rightarrow}{\Rightarrow}{\Rightarrow}}
\newcommand{\miff}{\mathchoice{\Longleftrightarrow}{\Leftrightarrow}{\Leftrightarrow}{\Leftrightarrow}}

\begin{document}
\author{Hannes Diener and Maarten McKubre-Jordens}
\title{Classifying Material Implications over Minimal Logic}
\maketitle

\abstract{The so-called paradoxes of material implication have motivated the development of many non-classical logics over the years \cite{aA75,nB77,aA89,gP89,sH96}. In this note, we investigate some of these paradoxes and classify them, over minimal logic. We provide proofs of equivalence and semantic models separating the paradoxes where appropriate. A number of equivalent groups arise, all of which collapse with unrestricted use of double negation elimination. Interestingly, the principle \emph{ex falso quodlibet}, and several weaker principles, turn out to be distinguishable, giving perhaps supporting motivation for adopting minimal logic as the ambient logic for reasoning in the possible presence of inconsistency.}

\providecommand{\keywords}[1]{\vskip 1pc \textbf{\textit{Keywords:}} \ \ #1.}
\providecommand{\acknowledgements}[1]{\textbf{\textit{Acknowledgements.}} \ \ #1.}



\section{Introduction}

The project of constructive reverse mathematics \cite{hI06a} has given rise to a wide literature where various theorems of mathematics and principles of logic have been classified over intuitionistic logic. What is less well-known is that the subtle difference that arises when the principle of explosion, \emph{ex falso quodlibet}, is dropped from intuitionistic logic (thus giving (Johansson's) \emph{minimal logic}) enables the distinction of many more principles. The focus of the present paper are a range of principles known collectively (but not exhaustively) as the \emph{paradoxes of material implication}; paradoxes because they illustrate that the usual interpretation of formal statements of the form ``$\ldots\implies\ldots$'' as informal statements of the form ``if\ldots then\ldots'' produces counter-intuitive results.

 Some of these principles were hinted at in \cite{sO08}. Here we present a carefully worked-out chart, classifying a number of such principles over minimal logic. These principles hold classically, and intuitionistically either hold or are equivalent to one of three well-known principles (see Section \ref{sec:intuitionism}). As it turns out, over minimal logic these principles divide cleanly into a small number of distinct categories. We hasten to add that the principles we classify here are considered as \emph{formula schemas}, and not individual instances. For example, when we write the formula $\neg\neg\varphi\implies\varphi$ for double negation elimination (\DNE), we mean that this should apply to all well-formed formulae $\varphi$. The work presented here is thus not a narrowly-focused investigation of what-instance-implies-what-instance, but rather a broad-stroke painting that classifies formula schemas as a whole.

This paper may be received in two ways: straightforwardly, as a contribution to reverse mathematics over non-classical logics; or more subtly as providing some insight into the kinds of distinctions that a good paraconsistent logic might contribute. Highlights of the paper include many refinements over \cite[chapter 6]{vV13}.

In what follows, we take: $\implies$ to be minimal implication; $\bot$ a logical constant not further defined (with the usual identification of $\neg\alpha$ with $\alpha\implies\bot$ for any well-formed formula $\alpha$); $\top$ a logical constant interchangeable with $\alpha\implies\alpha$ for some (arbitrary) well-formed formula $\alpha$ (that is, $\top$ is always satisfied). 

We will be interested in propositional axiom schemas, and instances of such schemas. A \emph{(propositional) axiom schema} is any well-formed formula, where the propositional variables are interpreted as ranging over well-formed formulas. An \emph{instance} of a formula schema is the schema with well-formed formulae consistently substituting for propositional variables in the schema in the intuitively obvious way. For example, the well-formed formula
$$\neg (\alpha \land \beta)  \implies (\alpha \land \beta \implies \neg \gamma)$$
is an instance of the formula schema
$$\neg \varphi  \implies (\varphi \implies \psi)$$
(\ref{PMI5} in what follows), where $\varphi$ is replaced with $\alpha\land\beta$, and $\psi$ is replaced with $\neg\gamma$. We also say that formula schema $\Phi$ \emph{implies} formula schema $\Psi$ if, given any instance $\psi$ of $\Psi$, there are finitely many instances $\{\varphi_1,\varphi_2,\dots,\varphi_n\}$ of $\Phi$ such that
$$\varphi_1,\varphi_2,\dots,\varphi_n \mimplies \psi,$$
where $\mimplies$ is understood as the (meta-theoretic) minimal logic consequence relation. We also say that $\Phi$ and $\Psi$ are equivalent if both $\Phi$ implies $\Psi$ and vice versa. Often it is theoretically useful to distinguish instance-wise implication from schematic implication \cite{hI16}, which is why we restrict ourselves to instance-wise proofs.

\section{Paradoxes of material implication}

The paradoxes we classify are the following schemas. Where $\varphi, \psi, \beta, \vartheta$ are any well-formed formulas,
\begin{enumerate} 
\item \label{PMI1} $(\varphi \implies \psi) \vee (\psi \implies \vartheta)$ \hfill (linearity, strong form)
\item \label{PMI2} $(\varphi \wedge \neg \varphi) \implies \psi$ \hfill (\emph{ex contradictione quodlibet}; the paradox of entailment)
\item \label{PMI3} $\varphi \implies (\psi \vee \neg \psi)$
\item  \label{PMI4} $ (\varphi \implies \psi) \vee \neg \psi$

  \item \label{PMI5} $\neg \varphi  \implies (\varphi \implies \psi) $
  \item \label{PMI6} $(\neg \varphi  \implies \varphi )\implies \varphi $ \hfill (\emph{consequentia mirabilis}; Clavius's law)\footnote{The version $( \varphi  \implies \neg \varphi )\implies \neg \varphi $ is used in \cite{rS95b}. However, since \emph{ex falso quodlibet falsum} ($\bot \implies \neg \phi$) holds in minimal logic, that version is provable over minimal logic. This suggests that while the present work illustrates many distinctions, there are still more distinctions that are not apparent here.}
  \item \label{PMI7}  $((\varphi \wedge \psi) \implies \vartheta) \implies \left( (\varphi \implies \vartheta) \vee (\psi \implies \vartheta) \right)$
\item \label{PMI8} $((\varphi \implies \vartheta) \wedge (\psi \implies \beta)) \implies \left( (\varphi \implies \beta) \vee ( \psi \implies \vartheta) \right) $
  \item \label{PMI9} $(\neg (\varphi \implies \psi)) \implies (\varphi \wedge \neg \psi)$ \hfill (the counterexample principle)
  \item \label{PMI10} $(((\varphi \implies \psi) \implies \varphi) \implies \varphi)$  \hfill (Peirce's Principle (\PP))
 \item \label{DGP} $(\varphi \implies \psi) \vee (\psi \implies \varphi)$  \hfill (Dirk Gently's Principle\footnote{Our name for this is based on the guiding principle of the protagonist of Douglas Adam's novel \emph{Dirk Gently's Holistic Detective Agency} \cite{dA87} who believes in the ``fundamental interconnectedness of all things.'' It also appears as an axiom in G\"odel–Dummett logic, and is more commonly known as (weak) linearity.} (\DGP))
 \item \label{PMI12} $(\neg \neg \varphi \vee \varphi) \implies \varphi$\footnote{The Wikipedia page at \url{https://en.wikipedia.org/wiki/Consequentia_mirabilis} actually lists this as equivalent to \ref{PMI6}, however that is not quite correct as we can see below. If we add the assumption that $\varphi \implies \psi \equiv \neg \varphi \vee \psi$, then they do turn out to be equivalent. However \emph{that} statement---interpreting $\implies$ as material implication---is minimally at least as strong as \LEM.}
 \item \label{PMI13}
      $\psi \vee (\psi \implies \vartheta)$  \hfill  (Tarski's formula)
\item \label{PMI14}
$ (\neg \varphi \implies \neg \psi) \vee (\neg \psi \implies \neg \varphi) $  \hfill  (weak Dirk Gently's Principle)
\item \label{PMI15}
$(\varphi \implies \psi \vee \vartheta) \implies ((\varphi \implies \psi) \vee (\varphi \implies \vartheta)) $
\item \label{PMI16} $ \neg (\varphi \implies \neg \varphi) \implies \varphi $ \hfill  (a form of Aristotle's law\footnote{Connexive logics are closely related. There, instead, the antecedent is taken as axiom schema; thus connexive logics are entirely non-classical, since $\neg(\varphi\implies\neg\varphi)$ is not classical valid.})

\end{enumerate}
Strictly speaking, these are of course axiom schemes rather then single axioms. These sentences are all classical theorems, where the arrow $\implies$ is interpreted as material implication.\footnote{Since all of these axioms are not provable in minimal logic, deducing them in classical logic is not mechanic and they are therefore good exercises for students.} Initially, we are interested how these principles relate to the following basic logical principles, the universal applicability of which is well-known to be rejected by intuitionistic (or, in the third case, minimal) logic:
\begin{description}
  \item[DNE] \label{DNE} $\neg \neg \varphi \implies \varphi$ \hfill (double negation elimination)
  \item[LEM] \label{LEM} $\varphi \vee \neg \varphi$ \hfill (law of excluded middle)
  \item[EFQ] \label{EFQ} $\bot \implies \varphi$ \hfill (ex falso quodlibet)\footnote{Of course this \emph{is} intuitionistically provable/definitional.}
  \item[WLEM] \label{WLEM} $\neg \varphi \vee \neg \neg \varphi$ \hfill (weak law of excluded middle)\footnote{An axiom in Jankov's logic, and De Morgan logic.}
\end{description}
It will turn out that there are two further important distinguishable classes; those related to Peirce's Principle, and those related to Dirk Gently's Principle.

We will also consider the following versions of De Morgan's laws

\begin{description}
	\item[DM1] \label{DM1}
 $\neg (\varphi \wedge \psi) \iff \neg \varphi \vee \neg \psi$
\item[DM2] \label{DM2} $\neg (\varphi \vee \psi) \iff \neg \varphi \wedge \neg \psi$
\item[DM1$^\prime$] \label{DM1p} $ \neg (\neg \varphi \wedge \neg \psi) \iff  \varphi \vee  \psi$
\item[DM2$^\prime$] \label{DM2p} $ \neg (\neg \varphi \vee \neg \psi) \iff  \varphi \wedge  \psi$
\end{description}

Notice that DM2 can be proven in minimal logic.

\section{Equivalents}

\begin{Pro}\label{pro:1}
The following are equivalent:
\begin{enumerate}[a)]
  \item \DNE, \ref{PMI9},  \ref{PMI12}, \ref{PMI16}, DM1$^\prime$, and DM2$^\prime$;
  \item \LEM, \ref{PMI3}, \ref{PMI4},  and \ref{PMI6};
  \item \EFQ, \ref{PMI2}, and \ref{PMI5};
  \item \WLEM, \ref{PMI14}, and DM1
\end{enumerate}
\end{Pro}
\begin{proof}
\begin{enumerate}[a)]
  \item Clearly $\DNE \miff \ref{PMI12}$ and $\DNE \miff \ref{PMI16}$. To see that $\DNE \mimplies \ref{PMI9}$, assume $\DNE$ and suppose that \begin{equation}\label{eq:dne2}\neg(\varphi\implies\psi).\end{equation} Suppose also, for contradiction, that \begin{equation}\label{eq:dne1}\neg(\varphi\land\neg\psi).\end{equation} Suppose further that $\varphi$ and $\neg\psi$. Adjunction then gives $\varphi\land\neg\psi$, contradicting \eqref{eq:dne1}; and so $\neg\neg\psi$ (discharging the assumption $\neg\psi$). $\DNE$ yields $\psi$, whence $\varphi\implies\psi$ (discharging $\varphi$). This contradicts \eqref{eq:dne2}, and so $\neg\neg(\varphi\land\neg\psi)$. Another application of $\DNE$ yields \ref{PMI9}.
  
   For the converse, suppose $\neg\neg\varphi$; that is, $\neg(\varphi\implies\bot)$. Then by \ref{PMI9}, $\varphi \land\neg\bot$. Hence, $\varphi$. 
   
   Clearly \DNE implies both versions of De Morgan's laws. Conversely we see that for $\psi = \varphi$ they both reduce to \DNE.
   \item $\LEM \miff \ref{PMI3}$, $\LEM \miff \ref{PMI4}$ and $\LEM \mimplies \ref{PMI6}$ are straightforward.\footnote{With judicious substitutions---in particular, using $\varphi \equiv \top$ in the right-to-left implications.}  To see that $\ref{PMI6}\mimplies\LEM$, it is enough to show that $\neg (\varphi \vee \neg \varphi) \implies \varphi \vee \neg \varphi$. So assume $\neg (\varphi \vee \neg \varphi)$. Then $\varphi$ leads to a contradiction, whence $\neg \varphi$ holds. This can be weakened to $\varphi \vee \neg \varphi$ and we are done.
 
  \item Clearly, $\EFQ \mimplies \ref{PMI2} \mimplies \ref{PMI5}$. To see that $\ref{PMI5}\mimplies\EFQ$, note that from $\bot$ we may deduce $\top\implies\bot$ by weakening; that is, $\neg\top$. Applying $\ref{PMI5}$, $\top\implies\varphi$, whence $\varphi$. The deduction theorem then yields $\EFQ$.
 
   \item First we show $\WLEM \miff \ref{PMI14}$. We have that either $\neg \psi$ or $\neg \neg \psi$ holds. In the first case we also have $\neg \varphi \implies \neg \psi$. In the second case assume that $\neg \psi$ holds. So we have $\bot$, which gives $\neg \varphi$ by weakening.

Conversely we either have $ \neg \varphi \implies \neg \neg \varphi$ or $\neg \neg \varphi \implies \neg \varphi$. In the first case the assumption that $\neg \varphi$ holds leads to $\bot$, and hence $\neg \neg \varphi$. In the second case, assuming $\varphi$, we get $\neg\neg\varphi$ and thus $\neg\varphi$, which gives $\bot$ by detachment; therefore, $\neg\varphi$.

Next we show $\WLEM \miff \text{DM1}$. Assume DM1 and let $\varphi$ be arbitrary. Since $\neg(\neg \varphi \wedge \varphi)$ is provable in intuitionistic logic we have $\neg \neg \varphi \vee \neg \varphi  $; that is \WLEM holds. 

Conversely assume that $\neg(\varphi \wedge \psi)$. By \WLEM either $\neg \varphi$ or $\neg \neg \varphi$. It is easy to see that in the second case the assumption that $\psi$ holds leads to a contradiction. Hence $\neg \psi$ and we are done.	

\end{enumerate}
\end{proof}
Next, we single out Peirce's Principle ($\ref{PMI10}$, \PP), and Dirk Gently's Principle ($\ref{DGP}$, \DGP). As will be shown in the next section, these principles are strictly distinguishable from the others.
\begin{Pro}\label{pro:2} The following are equivalent:
\begin{enumerate}[(a)]
\item \PP, \ref{PMI1}, and \ref{PMI13};
  \item \DGP, \ref{PMI7}, \ref{PMI8}, and \ref{PMI15}.
\end{enumerate}
\end{Pro}
\begin{proof}
\begin{enumerate}[(a)]
\item We show $\ref{PMI1} \mimplies \PP \mimplies \ref{PMI13} \mimplies \ref{PMI1}$.

First, consider $\varphi$, and $\psi$ such that $(\varphi \implies \psi) \implies \varphi$. By \ref{PMI1} either $\top \implies \varphi$ or $\varphi \implies \psi$. In the first case $\varphi$ holds. In the second case we can use the assumption to show that, also, $\varphi$ holds. Together
\[ ((\varphi \implies \psi) \implies \varphi) \implies \varphi \]
so that $\PP$ holds. 

Next, assume \PP, and let $ \psi$ and $\vartheta$ be arbitrary well-formed formulas. By \PP,
\[ (((\psi \vee (\psi \implies \vartheta )) \implies \vartheta) \implies \psi \vee (\psi \implies \vartheta ))\implies \psi \vee (\psi \implies \vartheta ). \] 
We show that the antecedent (and hence, by modus ponens, the consequent) of this holds. So assume $(\psi \vee (\psi \implies \vartheta )) \implies \vartheta)$. Furthermore assume $\psi$. Then $\vartheta$, so (discharging the assumption $\psi$) we have $\psi \implies \vartheta$. This weakens to $\psi \vee (\psi \implies \vartheta)$, so \ref{PMI13} follows. 

Last, since $\implies$ weakens, clearly \ref{PMI13} implies \ref{PMI1}.

\item First, we show $\DGP \miff \ref{PMI7}$. Assume that $(\varphi \wedge \psi) \implies \vartheta$. Now if \DGP holds then either $\varphi \implies \psi$ or $\psi \implies \varphi$.
	In the first case, if $\varphi$ holds, then also $\varphi \wedge \psi$, and hence $\vartheta$ holds. Together that means that in the first case we have $\varphi \implies \vartheta$. Similarly, in the second case $\psi \implies \vartheta$. Conversely, apply \ref{PMI7} to $\vartheta \equiv \varphi \wedge \psi$. Then the antecedent is $\top$, and so $ (\varphi \implies (\varphi \wedge \psi)) \vee (\psi \implies (\varphi \wedge \psi))$. Hence the desired $(\varphi \implies \psi) \vee (\psi \implies \varphi)$ holds.
	
	Next, we show $\DGP \miff \ref{PMI8}$. Assume that $(\varphi \implies \vartheta) \wedge (\psi \implies \beta)$. By \DGP either $\varphi \implies \beta$ and we are done, or $\beta \implies \varphi$. But in that second case if we assume $\psi$ also $\beta$ holds, which in turn implies $\varphi$, which in turn implies $\vartheta$. Together, in the second case, $\psi \implies \vartheta$. Conversely, apply \ref{PMI8} to $\vartheta \equiv \varphi$ and $\beta \equiv \psi$, which yields 
\[((\varphi \implies \varphi) \wedge (\psi \implies \psi)) \implies \left( (\varphi \implies \psi) \vee ( \psi \implies \varphi) \right)  \ . \] 
Since the antecedent is $(\top \land \top) \equiv \top$, we get the desired $(\varphi \implies \psi) \vee (\psi \implies \varphi)$.
	
	Last, $\DGP \miff \ref{PMI15}$. For the forward direction, suppose that $\varphi\implies\psi\lor\vartheta$. $\DGP$ gives $(\psi\implies\vartheta)\lor(\vartheta\implies\psi)$. In the first case, assuming $\varphi$, we get $\psi\lor\vartheta$, which (by modus ponens on $\psi\implies\vartheta$) is $\vartheta$ in this case. So $\varphi\implies\vartheta$, which weakens to $(\varphi\implies\psi)\lor(\varphi\implies\vartheta)$. In the second case, a similar argument also shows $(\varphi\implies\psi)\lor(\varphi\implies\vartheta)$. Either way the consequent of $\ref{PMI15}$ holds; whence, by the deduction theorem, $\ref{PMI15}$. Conversely, apply \ref{PMI15} to $\varphi \vee \psi$, $\varphi$, and $\psi$ to get:
	\[ (\varphi \vee \psi \implies \varphi \vee \psi) \implies ((\varphi \vee \psi \implies \varphi ) \vee (\varphi \vee \psi \implies \psi)) \ .\] Now clearly the antecedent is always satisfied. So we have that 
	\[ (\varphi \vee \psi \implies \varphi ) \vee (\varphi \vee \psi \implies \psi) \ , \] which is equivalent to the desired $(\psi \implies \varphi ) \vee (\varphi \implies \psi) $. 
\end{enumerate}
\end{proof}

%
%

We are now in a position to lay out how these principles fit together.
\begin{Pro}
The implications in Figure \ref{fig:1} hold:
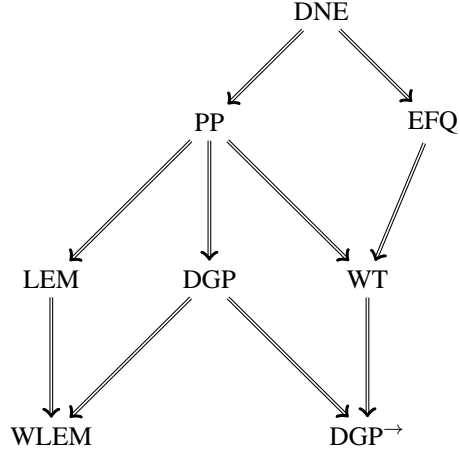
\begin{figure}[ht]\centering
\begin{tikzpicture}[node distance=.1\textheight, auto]
  \node (DNE) {\DNE};
  \node (PP) [below left of=DNE] {\PP}; 
  \node (EFQ) [below right of=DNE] {\EFQ};
  \node (DGP) [below of=PP] {\DGP};
  \node (LEM) [left of=DGP] {\LEM};
  \node (WLEM) [below  of=LEM] {\WLEM};
  \node (WT) [right of=DGP] {\WT};
  \node (DGPi) [below of=WT] {\DGPimp};

  \draw[double,->] (PP) to node {} (DGP);
  \draw[double,->] (DNE) to node {} (PP);
  \draw[double,->] (PP) to node {} (LEM);
  \draw[double,->] (DGP) to node {} (WLEM);
  \draw[double,->] (LEM) to node {} (WLEM);
  \draw[double,->] (DNE) to node {} (EFQ);
  \draw[double,->] (EFQ) to node {} (WT);
  \draw[double,->] (PP) to node {} (WT);
  \draw[double,->] (WT) to node {} (DGPi);
  \draw[double,->] (DGP) to node {} (DGPi);
\end{tikzpicture}
\caption{Some principles distinguishable over minimal logic. As is shown in Section \ref{sec:semantics}, none of the arrows can be reversed and no arrow can be added.}\label{fig:1}
\end{figure}
\end{Pro}
That $\DNE \mimplies \EFQ$ and $\LEM \mimplies \WLEM$ is clear; we prove the remaining implications.
\begin{proof}
\begin{itemize}
\item $\DNE \mimplies \PP$: Assume $(\varphi\implies\psi)\implies\varphi$, and $\neg\varphi$. Modus tollens gives $\neg(\varphi\implies\psi)$ whence, by counterexample (\ref{PMI9}) (see Proposition \ref{pro:1}), $\varphi\land\neg\psi$; so $\varphi$. With $\neg\varphi$, this gives $\bot$; and hence (discharging the assumption $\neg\varphi$) $\neg\neg\varphi$. Appying $\DNE$ gives $\varphi$.
\item $\PP \mimplies \LEM$: Assume $\neg(\varphi\lor\neg\varphi)$; that is, $(\varphi\lor\neg\varphi) \implies \bot$. Then $\varphi$ leads to a contradiction; so $\neg\varphi$. But then $\varphi\lor\neg\varphi$. So we have $((\varphi\lor\neg\varphi)\implies\bot)\implies(\varphi\lor\neg\varphi)$. Applying \PP, $\varphi\lor\neg\varphi$.
\item $\PP \mimplies \DGP$: By \PP,
$$\left(\left(\left(\left(\varphi\implies\psi\right)\lor\left(\psi\implies\varphi\right)\right) \implies \varphi \right) \implies \left(\left(\varphi\implies\psi\right)\lor\left(\psi\implies\varphi\right)\right)\right) \implies \left(\left(\varphi\implies\psi\right)\lor\left(\psi\implies\varphi\right)\right).$$
We show that the antecedent holds. Assume 
\begin{equation}\label{eq:ass1}\left(\left(\varphi\implies\psi\right)\lor\left(\psi\implies\varphi\right)\right) \implies \varphi,\end{equation}
and suppose $\psi$. Then $\varphi\implies\psi$, so by \eqref{eq:ass1}, $\varphi$. Thus (discharging the assumption of $\psi$), $\psi\implies\varphi$. But then again by \eqref{eq:ass1}, $\varphi$, which weakens to $\left(\varphi\implies\psi\right)\lor\left(\psi\implies\varphi\right)$ and we are done.
\item $\DGP \mimplies \WLEM$: By \DGP, we have $(\varphi \implies \neg\varphi) \lor (\neg\varphi \implies \varphi)$. In the former case, assuming $\varphi$ gives $\bot$, whence $\neg\varphi$. In the latter case, assuming $\neg\varphi$ gives $\bot$, whence $\neg\neg\varphi$.
\end{itemize}
\end{proof}

\section{The implicational fragment} \label{Sec:impl_fragment}

For various technical reasons it is often interesting to work only with formulas built up from propositional symbols including $\bot$ with $\implies$. Of course, we still use $\neg$ as an abbreviation for $\implies \bot$. Assuming classical logic (\DNE) this is no restriction, since there $\vee$ and $\wedge$ are definable from $\implies$ and $\neg$. More precisely we can define $\varphi \vee \psi$ as $\neg \varphi \implies \psi$, but over minimal logic this validates \EFQ.\footnote{Simply use $\lor$-introduction and the proposed translation; \EFQ follows. To translate $\varphi\lor\psi$ as $\neg\varphi\implies\psi$ would therefore be disingenuous.} A more faithful (but slightly weaker) translation is:
\begin{equation}\label{eq:veetrans}
\varphi \vee \psi := (\varphi \implies \bot) \implies (\psi \implies \bot) \implies \bot \qquad \left[ \equiv \neg \varphi \implies \neg\neg \psi \right].
\end{equation}
Notice that we might also translate $\varphi \vee \psi$ as $\neg \psi \implies \neg\neg \varphi$, whose equivalence to \eqref{eq:veetrans} is minimally provable; we will use whichever of the two translations is more expedient. Moreover, conjunction can be removed entirely also: $\varphi\land\psi\implies\vartheta$ translates to $\varphi\implies\psi\implies\vartheta$,\footnote{Note that negated conjunction is a special case.} and $\vartheta\implies\varphi\land\psi$ translates to the two separate cases $\vartheta\implies\varphi$ and $\vartheta\implies\psi$.\footnote{It is not clear, however, that this makes no difference in \emph{proofs}; more on this later.} Translates of formulas are then defined in the obvious way. We denote the translation of a formula $\varphi$ into the implication-only fragment by $\varphi^\implies$.

The following are of special note:
\begin{itemize}
  \item \ref{PMI13}, a strong form of linearity, which (by Proposition \ref{pro:2} is equivalent to \DNE. Its translation into implicative form is:
  \begin{equation}\tag{Weak Tarski's Formula, \WT}
  \neg \psi \implies \neg \neg (\psi \implies\vartheta) \ ,
  \end{equation}
  which is an abbreviation for $(\psi\implies\bot)\implies((\psi\implies\vartheta)\implies\bot)\implies\bot$.
  
  \item $\ref{DGP}$, Dirk Gently's Principle, which translates to:
  \begin{equation}\tag{Implicative Dirk Gently's Principle, \DGPimp}
  \neg(\varphi\implies\psi)\implies\neg\neg(\psi\implies\varphi).
  \end{equation}
\end{itemize}
The above principles are closely related, but distinct. A separation result can be found in Section \ref{sec:semantics}. 

\begin{Pro} The following implications hold:
\begin{enumerate}[(a)]
  \item $\PP \mimplies \WT$
  \item $\DGP \mimplies \DGPimp$
  \item $\EFQ \mimplies \WT$
  \item $\WT \mimplies \DGPimp$
\end{enumerate}
\end{Pro}
\begin{proof}
\begin{enumerate}[(a)]
\item We use \PP in the form $((\bot\implies\vartheta)\implies\bot)\implies\bot$. For modus ponens we need to establish that $(\bot\implies\vartheta)\implies\bot$. For the purpose of applying the deduction theorem to show this (and also \WT), assume:
\begin{enumerate}[(i)]
\item $\psi\implies\bot$; 
\item $(\psi\implies\vartheta)\implies\bot$; and
\item $\bot\implies\vartheta$.
\end{enumerate} 
Then by transitivity on (i) and (iii), $\psi\implies\vartheta$. Using (ii), $\bot$. The deduction theorem (discharging (iii)) yields $(\bot\implies\vartheta)\implies\bot$. Applying \PP gives $\bot$, whence (discharging (ii) in another application of the deduction theorem) $((\psi\implies\vartheta)\implies\bot)\implies\bot$. The conclusion follows by yet another application of the deduction theorem.
\item Assume that $\neg (\varphi \implies \psi) $. By \DGP either $\varphi \implies \psi$ or $\psi \implies \varphi$. In the first case we get $\bot$ by modus ponens, and therefore also $\neg \neg ( \psi \implies \varphi)$. In the second case, since minimally $\alpha \implies \neg \neg \alpha$, also $\neg \neg ( \psi \implies \varphi)$. Thus in both cases the conclusion holds.
\item The proof is similar to (a), but simpler, since (iii) is now no longer an assumption but an instance of \EFQ (and so we do not need to explicitly apply \PP).

\item Assume that $\neg(\varphi\implies\psi)$. Then, assuming $\psi$ and weakening leads to a contradiction; so $\neg\psi$. By \WT, $\neg\neg(\psi\implies\varphi)$. The conclusion follows from the deduction theorem.
\end{enumerate}
\end{proof}
In Section \ref{sec:semantics} we show that these implications are strict. For the paradoxes of material implication, using the translation \eqref{eq:veetrans} (and the comments following it), we have:
\begin{itemize}
  \item $\LEM^\implies$ is $\neg\varphi\implies\neg\neg\neg\varphi$, which is provable in minimal logic. Similarly, $\WLEM^\implies$ is provable.
  \item $\ref{PMI1}^\implies$ is $\neg (\varphi \implies \psi) \implies \neg \neg (\psi \implies \vartheta)$.
  \item $\ref{PMI2}^\implies$ is $\varphi\implies\neg\varphi\implies\psi$.
  \item $\ref{PMI3}^\implies$ is $ \varphi \implies \neg \psi \implies \neg\neg\neg\psi $, a weakening of double negation introduction and hence provable.
  \item $\ref{PMI4}^\implies$ is $ \neg (\varphi \implies \psi) \implies\neg\neg \neg \psi$, and is provable.
  \item $\ref{PMI7}^\implies$ is $(\varphi \implies \psi \implies \vartheta) \implies  \neg (\varphi \implies \vartheta) \implies\neg\neg (\psi \implies \vartheta) $.
  \item $\ref{PMI8}^\implies$ is $(\varphi \implies \vartheta) \implies (\psi \implies \beta) \implies \neg (\varphi \implies \beta) \implies\neg\neg ( \psi \implies \vartheta)$.
  \item $\ref{PMI9}^\implies$ is (a) $\neg (\varphi \implies \psi) \implies \varphi$ and (b) $\neg (\varphi \implies \psi) \implies \neg \psi$. The latter is provable.
  \item $\ref{DGP}^\implies$ is $\DGPimp$.
  \item $\ref{PMI12}^\implies$ is $(\neg \neg \neg \varphi \implies \neg\neg\varphi) \implies \varphi $.
  \item $\ref{PMI13}^\implies$ is \WT.
  \item $\ref{PMI14}^\implies$ is $ \neg(\neg \varphi \implies \neg \psi) \implies \neg\neg(\neg \psi \implies \neg \varphi) $. This is provable: assume $ \neg(\neg \varphi \implies \neg \psi)$ and $\neg\psi$. Weakening gives $\neg\varphi\implies\neg\psi$, a contradiction; whence $\neg\varphi$. So $\neg\psi\implies\neg\varphi$, which by double negation introduction gives the conclusion.
  \item $\ref{PMI15}^\implies$ is $(\varphi \implies \neg\psi \implies \neg\neg\vartheta) \implies \neg(\varphi \implies \psi) \implies\neg\neg (\varphi \implies \vartheta) $.
  \item DM$1^\implies$ is $(\varphi \implies \neg\psi) \iff (\neg\neg \varphi \implies \neg\neg\neg \psi)$, which is provable.
  \item DM$2^\implies$ is $\neg (\neg\varphi \implies\neg\neg \psi) \implies \neg \varphi$, $\neg (\neg\varphi \implies\neg\neg \psi) \implies \neg \psi$, and $\neg\varphi\implies\neg\psi\implies\neg (\neg\varphi \implies\neg\neg \psi)$, each of which is provable.
  \item DM$1^{\prime\implies}$ is $ (\neg \varphi \implies \neg\neg \psi) \iff  (\neg \varphi \implies\neg\neg  \psi)$, which is $\top$ (always satisfied).
  \item DM$2^{\prime\implies}$ is (a) $ \neg (\neg\neg \varphi \implies\neg\neg \neg \psi) \implies  \varphi$, (b) $ \neg (\neg\neg \varphi \implies\neg\neg \neg \psi) \implies \psi$, and  (c) $ \varphi\implies\psi\implies\neg (\neg\neg \varphi \implies\neg\neg \neg \psi)$. The latter, (c), is provable.
\end{itemize}
It is often technically useful to know when an operator may be pulled back through an implication, so two further sentences of interest are:
\begin{enumerate}\setcounter{enumi}{16}
\item \label{PMI17} $(\varphi\implies\neg\neg\psi)\implies\neg\neg(\varphi\implies\psi)$.
\item \label{PMI18} $\neg\neg(\varphi\implies\psi)\implies(\neg\neg\varphi\implies\psi)$.
\end{enumerate}
Note:
\begin{itemize}
\item The converse of \ref{PMI17} is provable. For, assume $\neg\neg(\varphi\implies\psi)$. Further assume
\begin{enumerate}[(i)]
\item $\varphi$;
\item $\neg\psi$; and
\item $\varphi\implies\psi$.
\end{enumerate}
(i) and (iii) lead to $\bot$; whence (discharging (iii)) $\neg(\varphi\implies\psi)$. But then $\bot$ again, so (discharging (ii)) $\neg\neg\psi$. Applying the deduction theorem twice gives the converse of \ref{PMI17}.
\item Likewise, the converse of $\ref{PMI18}$ is provable. Assume $\neg\neg\varphi\implies\psi$ and $\varphi$. Then $\neg\neg\varphi$, so $\psi$ and hence, in fact, $\varphi\implies\psi$, which is stronger than the converse of \ref{PMI18}.
\end{itemize}
We now turn to classifying the foregoing sentences, whenever they are not provable in minimal logic alone.
\begin{Pro} The following are equivalent:
\begin{enumerate}[(a)]
  \item $\DNE$, $\ref{PMI9}^\implies$(a), $\ref{PMI12}^\implies$, DM$2^{\prime\implies}$(a), DM$2^{\prime\implies}$(b), $\ref{PMI18}$.
  \item $\EFQ$, $\ref{PMI2}^\implies$.
  \item $\WT$, $\ref{PMI1}^\implies$, $\ref{PMI15}^\implies$, $\ref{PMI17}$.
\end{enumerate}
\end{Pro}
\begin{proof}
\begin{enumerate}[(a)]
  \item Since full classical logic is obtained by adding \DNE to minimal logic, it suffices to prove that each numbered principle implies \DNE.
  
   Observe that \DNE is a special case of $\ref{PMI9}^\implies$(a)---namely, the case where $\psi = \bot$.
  
  To see that $\ref{PMI12}^\implies \mimplies \DNE$, assume $\neg\neg\varphi$. By weakening, $\neg\neg\neg\varphi\implies\neg\neg\varphi$, and so by $\ref{PMI12}^\implies$, $\varphi$.
  
  To see that  DM$2^{\prime\implies}$(a), DM$2^{\prime\implies}$(b) each imply \DNE, substitute $\varphi$ for $\psi$; then, assuming $\neg\neg\varphi$, in each case the antecedent is satisfied, and in each case the consequent is $\varphi$.
  
  To see that \ref{PMI18} implies \DNE, substitute $\varphi := \psi$; then the antecedent of \ref{PMI18} is always satisfied, and the consequent is \DNE.
  
  \item It is clear that $\EFQ\mimplies \ref{PMI2}^\implies$. For the converse, suppose we wish to show that $\bot\implies\vartheta$. Assume $\bot$. Then (weakening) $\neg\top$. In particular, substituting $\top$ for $\varphi$ and $\vartheta$ for $\psi$ in $\ref{PMI2}^\implies$ gives
  $$\top\implies\neg\top\implies\vartheta.$$
  Two applications of modus ponens followed by the deduction theorem gives the desired conclusion.
  
  \item Since $\neg(\varphi\implies\psi) \implies\neg\psi$ is provable, transitivity with \WT shows that $\WT \mimplies \ref{PMI1}^\implies$. Conversely, suppose that $\neg\psi$, and $\neg(\psi\implies\vartheta)$. In view of the deduction theorem, we aim to derive $\bot$. Consider the following form of $\ref{PMI1}^\implies$:
  $$\neg(\neg\psi\implies\psi) \implies \neg\neg(\psi\implies\vartheta).$$
  Using the assumption $\neg(\psi\implies\vartheta)$, by the provable version of contraposition we conclude $\neg\neg(\neg\psi\implies\psi)$. Since the converse of \ref{PMI17} is provable, $\neg\psi\implies\neg\neg\psi$. But then, using the assumption $\neg\psi$, $\neg\neg\psi$, which gives $\bot$ and we are done.
  
  To see that $\WT^\implies \mimplies \ref{PMI15}^\implies$, assume:
  \begin{enumerate}[(i)]
  \item $\varphi \implies (\neg\psi \implies \neg\neg\vartheta)$;
  \item $\neg(\varphi \implies \psi)$; and
  \item $\neg (\varphi \implies \vartheta)$.T
  \end{enumerate}
  In light of the deduction theorem, it is enough to prove $\bot$. Suppose $\varphi$. Using (i) now yields $\neg\psi\implies\neg\neg\vartheta$. From (ii) we see $\neg\psi$; whence $\neg\neg\vartheta$. But from (iii), $\neg\vartheta$, a contradiction. Therefore, $\neg\varphi$. Applying $\WT$, $\neg\neg(\varphi\implies\psi)$, which contradicts (ii). Hence $\bot$, and we are done.
  
  Conversely, suppose that $\neg\psi$ and (for the purpose of deriving a contradiction) $\neg(\psi\implies\vartheta)$. Suppose that $\psi$. Then $\bot$, so that $\neg\neg\vartheta$, which weakens to $\neg\vartheta\implies\neg\neg\vartheta$. Discharging the assumption, $\psi\implies \neg\vartheta\implies\neg\neg\vartheta$. Then invoke $\ref{PMI15}^\implies$ in the form
  $$(\psi\implies\neg\vartheta\implies\neg\neg\vartheta) \implies (\neg(\psi\implies\vartheta)\implies\neg\neg(\psi\implies\vartheta))$$
  and detach twice (using the assumptions still in play) to get $\neg\neg(\psi\implies\vartheta)$. This contradicts the assumption, so $\bot$. Hence \WT.\footnote{It may be interesting to note that at least one instance of \emph{contraction} is used in this proof.}
  
  To see that $\WT^\implies \mimplies \ref{PMI17}$, suppose that $\varphi\implies\neg\neg\psi$, and (for contradiction) that $\neg(\varphi\implies\psi)$. From the latter, $\neg\psi$; whence, using the former, $\neg\varphi$. Applying $\WT$, $\neg\neg(\varphi\implies\psi)$, which contradicts the assumption; so $\bot$.
  
  Conversely, suppose that $\neg\psi$. Assuming $\psi$, we get $\bot$, so we may conclude $\neg\neg\vartheta$ and therefore $\psi\implies\neg\neg\vartheta$. Applying \ref{PMI17}, $\neg\neg(\psi\implies\vartheta)$ and we are done.

\end{enumerate}
\end{proof}
\begin{Pro}
The following implications hold:
\begin{enumerate}[(a)]
  \item $\WT \mimplies \ref{PMI7}^\implies$;
  \item $\DGPimp \miff \ref{PMI8}^\implies$;
  \item $\ref{PMI7}^\implies \mimplies \DGPimp$ (using $\wedge$). 
\end{enumerate}
\end{Pro}
\begin{proof}
\begin{enumerate}[(a)]
  \item  Assume that $\varphi\implies\psi\implies\vartheta$. By interderivability of $\neg\alpha \implies \neg\neg\beta$ and $\neg \beta \implies\neg\neg \alpha$, and in view of the deduction theorem, we may show that $$\neg(\psi\implies\vartheta) \implies \neg\neg(\varphi\implies\vartheta).\footnote{Substitute $\varphi\implies\vartheta$ for $\alpha$, and $\psi\implies\vartheta$ for $\beta$}$$ Assume the antecedent. Transitivity over the assumptions gives $\neg\varphi$. Applying \WT, $\neg\neg(\varphi\implies\vartheta)$ and we are done.
  \item First consider $\varphi$, $\psi$, $\vartheta$, and $\beta$ such that $\varphi \implies \vartheta$, $\psi \implies \beta$, and  $\neg \varphi \implies \beta$. We want to show that $\neg\neg ( \psi \implies \vartheta)$. By $\DGPimp$ it is enough to show $\neg (\vartheta \implies \psi)$. So assume also $\vartheta \implies \psi$, which together with the above assumption implies $\varphi \implies \beta$, but that contradicts the third of the initial assumptions. Thus we can derive $\bot$ and are done.

Conversely, a special case of $\ref{PMI8}^\implies$ is 
\[ (\varphi \implies \varphi) \implies (\psi \implies \psi) \implies \neg (\varphi \implies \psi) \implies\neg\neg ( \psi \implies \varphi)  \ .\]
Since the first to assumptions are tautologies we have
\[  \neg (\varphi \implies \psi) \implies\neg\neg ( \psi \implies \varphi)  \ ,\] which is $\DGPimp$.
\item Let $\varphi$ and $\psi$ be such that $\neg (\varphi \implies \psi)$ and $\neg (\psi \implies \varphi)$. Notice that then also $\neg (\varphi \implies \varphi \wedge \psi )$ and $\neg (\psi \implies \varphi \wedge \psi )$. We want to derive at a contradiction.

Applying  $\ref{PMI7}^\implies$ to $\varphi$, $\psi$, and $\varphi \wedge \psi$ gives us 
\[ (\varphi \implies (\psi \implies (\varphi \wedge \psi))) \implies \neg (\varphi \implies (\varphi \wedge \psi)) \implies \neg \neg (\psi \implies (\varphi \wedge \psi)) \ . \]
Notice that the antecedent is provable minimally, so we have
\[ \neg (\varphi \implies (\varphi \wedge \psi)) \implies  \neg (\psi \implies (\varphi \wedge \psi)) \implies \bot \ , \]
and therefore, applying this to out assumptions we get the desired $\bot$.

\end{enumerate}
\end{proof}

We will comment on the strange status of statement $\ref{PMI7}^\implies$ in the last section.

\section{Separation results (semantics)}\label{sec:semantics}
To show the strictness of the implications summed up in Figure \ref{fig:1} we will use models rather than proof theoretic methods, which is the route taken in \cite{vV13}. We base our semantics for minimal logic  on the one described in \cite{sO08}. More precisely, we consider $(\WW,\sqsubseteq,Q)$ where $(\WW,\sqsubseteq)$ is a partial order and $Q \subseteq \WW$ is a \define{cone}---that is, an upwards closed set. A \define{valuation} $v$ is a monotone mapping from $\WW$ to $\Pow{\Prop}$. We will call the elements of $\WW$ \define{worlds}.

A \define{model} is a pair $(\WW,v)$ and the forcing relation between a model and a formula is defined in almost the same way as for Kripke semantics. That is we set
\begin{align*}
	u \Vdash P \miff P \in v(u) 
\end{align*} 
for propositional formulas and then inductively
\begin{align*}
	u \Vdash \varphi \wedge \psi & \miff u \Vdash \varphi \text{ and } u \Vdash \psi \\
	u \Vdash \varphi \vee \psi & \miff u \Vdash \varphi \text{ or } u \Vdash \psi \\
	u \Vdash \varphi \implies \psi & \miff \fa{y \in \WW}{\left(u \sqsubseteq y \wedge y \Vdash \varphi \mimplies y \Vdash \psi \right)}  \ .
\end{align*} 
A point of difference with the usual Kripke semantics is that we do not assume that $\bot$ is never forced, but that we have 
\[ u \Vdash \bot \miff u \in Q  \ .\]
The intuition behind Kripke style semantics is that we have a multitude of possible worlds ordered by $\sqsubseteq$ with the requirement that if a formula is true in some world it is also true in worlds ``above'' it relative to the order. Each world by itself behaves like a classical model---of course, apart from our treatment of $\bot$. We will call worlds $u$ such that $u \in Q$ \define{abnormal}, and the otherwise \define{normal}. Since we have defined $\neg \varphi$ as $\varphi \implies \bot$,\footnote{This is a point of departure for many other non-classical logics, such as relevant logics, logics of formal inconsistency, and the like. As mentioned, this is a preliminary investigation into the realm of non-classical reverse mathematics more generally; so we stick fairly close to the usual, classical, interpretation of negation.} we get
\[ u \Vdash \neg \varphi  \miff \fa{y \in \WW}{\left(u \sqsubseteq y \wedge y \Vdash \varphi \mimplies y \in Q \right)}  \ , \]
that is $\neg \varphi$ holds in some world  if the only worlds above in which $\varphi$  holds are abnormal ones. 

We write $(\WW,v) \Vdash \varphi$ or simply $\WW \Vdash \varphi$, if $\varphi$ is forced in all worlds of $\WW$. As usual, the valuation $v$ should be clear from the context and will not be mentioned. There is another reason why $v$ for us is actually irrelevant: in the present paper, we only consider \define{full models}. That is, we assume that for every upward closed set $U \subseteq \WW$ there exists a propositional symbol $P_U$ such $u \Vdash P_U \miff u \in U$. 
This excludes many pathological cases of considering complicated structures with trivial valuations. It also means that for any well formed formula $\alpha$ there exists a propositional symbol $P_\alpha$ such that 
\[ \WW \Vdash \alpha \miff \WW \Vdash P_\alpha \ ,\]
which has the advantage, that we only need to consider our axiom schemes to range over propositional symbols and not arbitrary formulas.

Having an arbitrary partial order as our underlying structure is different from \cites{hI16, kW01} which uses tree like-orders---or to be precise finitely branching trees. There is a subtle difference between the two notions as we will explain. On first glance the differences seem minuscule:
\begin{Pro}
Assume that $(\WW,v)$ is a model that is tame in the sense that if $L$ is a maximal chain with maximal element $u$, and $u \sqsubseteq v$, then there exists some maximal chain $L^\prime$ having $v$ as a maximal element and $L \subset L^\prime$.\footnote{This holds if, for example, $\WW$ is finite, or one assumes Zorn's Lemma.}

Then there exists a tree-like model $(\WW^t,v^t)$ such that 
\begin{equation} \label{Eqn:modeltransform}
	\WW \Vdash \alpha \iff \WW^t \Vdash \alpha 
\end{equation} 
for any formula $\alpha$. Furthermore, if $(\WW,v)$ is finite (i.e.\ $\WW$ is finite as a set) then  $\WW^t$ is a finite and finitely branching tree.
\end{Pro}
\begin{proof}
	Define $\WW^t$ to be 
	\[ \set{(u,L) \in \WW \times \mathcal{P}(\WW)}{L \text{ is a maximal chain with } u \text{ as a maximal element}  } \ , \]
	and set \[ (u,L) \sqsubseteq (v,L^\prime) \miff L \subset L^\prime \ .\]
	Notice that if $(u,L) \sqsubseteq (v,L^\prime)$ then $u \in L^\prime$ and that, in particular, $u \sqsubseteq v$. As our new valuation set $v^t(u,L) = v(u)$.
	
	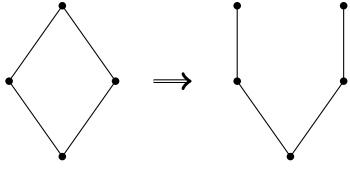
\begin{figure}[h!]
		\begin{tikzpicture}
  \node [circle,fill=black,inner sep=1pt] at (0,0) {};
  \node [circle,fill=black,inner sep=1pt] at (-0.7,1) {};
  \node [circle,fill=black,inner sep=1pt] at (0.7,1) {};
  \node [circle,fill=black,inner sep=1pt] at (0,2) {};
  \draw (0,0) -- (-0.7,1);
  \draw (0,0) -- (0.7,1);
  \draw (0,2) -- (-0.7,1);
  \draw (0,2) -- (0.7,1);

  \draw[->, double ] (1.2,1) -- (1.7,1);
  \node [circle,fill=black,inner sep=1pt] at (3,0) {};
  \node [circle,fill=black,inner sep=1pt] at (2.3,1) {};
  \node [circle,fill=black,inner sep=1pt] at (3.7,1) {};
  \node [circle,fill=black,inner sep=1pt] at (2.3,2) {};
  \node [circle,fill=black,inner sep=1pt] at (3.7,2) {};

  \draw (3,0) -- (2.3,1);
  \draw (3,0) -- (3.7,1);
  \draw (2.3,2) -- (2.3,1);
  \draw (3.7,2) -- (3.7,1);
\end{tikzpicture}

\caption{The construction of this proposition removes any joins.} \label{Fig:modeltransf}
	\end{figure}

	We will proof that for any formula $\alpha$ for all $(u,L) \in \WW^t$
	\begin{equation} \label{Eqn:tree_model}
		(u,L) \Vdash \alpha \miff u \Vdash \alpha		
	\end{equation}
	 by induction on the formula $\alpha$. 
	 If $\alpha$ is a propositional symbol (including $\bot$) then we have \ref{Eqn:tree_model} by our definition of $v^t$. The connectives $\wedge$ and $\vee$ are straightforward to deal with. So let $\alpha \equiv \beta \implies \gamma$. 

	 For the  direction ``$\Longleftarrow$'' fix $(u,L) \in \WW^t$ and assume that $u \Vdash \beta \implies \gamma$. Now consider $(v, L^\prime)$ such that $(u,L) \sqsubseteq (v,L^\prime)$ and that $ (v, L^\prime) \Vdash \beta$. By our induction hypothesis that means $v \Vdash \beta$, which implies that $v \Vdash \gamma$. Using the induction hypothesis again we get the desired $(v,L^\prime) \Vdash \gamma$.
	   
	  For the  direction ``$\Longrightarrow$'' assume that $(u,L) \Vdash \beta \implies \gamma$, and that $v$ is such that $u \sqsubseteq v$  and $v \Vdash \beta$. The only really non-trivial step in this entire proof is to use the tame-ness assumption to find $L^\prime$ which has $v$ as a maximal element and is such that $L \subseteq L^\prime$. Then $(u,L) \sqsubseteq (v,L^\prime)$. 
	  By our induction hypothesis we have that $(v,L^\prime) \Vdash \beta$, which means that $(v,L^\prime) \Vdash \gamma$. Using our induction hypothesis yet again this means that $v \Vdash \gamma$ and we are done.
\end{proof}
The differences between tree-like and non-tree-like models stem from our definition of a full model. Notice that even if $\WW$ is full $\WW^t$ might not be. In the example sketched in Figure \ref{Fig:modeltransf}, in $\WW^t$ there is no proposition that is forced at one of the top nodes, but not the other; all propositions are either forced at both top-nodes at the same time or not forced at both nodes.

If one does prefer to work with tree-like models one cannot restrict to fullness. For example, as one can easily see, any intuitionistic (i.e.\ $Q = \emptyset$) tree-like model containing a branching, i.e. that is not $\mathsf{v}$-free, does not satisfy \WLEM, which means that \emph{intuitionistic} tree-like models cannot distinguish between \WLEM and \DGP.\footnote{Nor can they distinguish between \DNE, \LEM, and \PP (see Proposition \ref{pro:pw-incomp}), but these are all known to be intuitionistically equivalent anyway; see Sec. \ref{sec:intuitionism}.} So for a structural analysis of what principles hold depending on the underlying partial order it makes more sense to consider arbitrary partial orders rather than tree-like ones. This also excludes Veldman's explosive nodes \cite{wV75}, which do not add further distinctions here.

It is straightforward to see that we have soundness \cite[Proposition 2.3.2]{sO08}, which means we can use these models to show the underivability of formulas in minimal logic. Models for intuitionistic logic, that is minimal logic together with \EFQ, are exactly the ones where $\bot$ is never forced (in this case we recover the usual Kripke semantics):
\begin{Pro} \label{Pro:EFGiffQempty}
$\WW \Vdash \EFQ$ if and only if $Q = \emptyset$.	
\end{Pro}
\begin{proof}
	Clearly, if $Q = \emptyset$ then $\WW \Vdash \EFQ$. Conversely, if there is an abnormal world $u \in Q$, then by fullness we can consider the propositional symbol $P_\emptyset$ for which $u \Vdash \bot$, but $u \nVdash P_\emptyset$; i.e.\ $u \nVdash \bot \implies P_\emptyset$. 
\end{proof}

Given a structure $(\WW,\sqsubseteq,Q)$ we define its \define{loBOTomy} $\WW^\bot$ to be $(\WW,\sqsubseteq,\WW)$, that is we are making all worlds are abnormal. Even this quite trivial construction has very useful consequences for us.
\begin{Pro} \label{Pro:loBOTomy}
	For any (full) model we have
	\begin{enumerate}
  \item $\WW^\bot \Vdash \LEM$
  \item $\WW^\bot \nVdash \DNE$ 
  \item $\WW^\bot \nVdash \EFQ$ 
\end{enumerate}
\end{Pro}
\begin{proof} The proofs are easy after one notices that $\WW^\bot \Vdash \neg \alpha$ for \emph{any} formula $\alpha$. It is also worth pointing out that we need the fullness assumption for the second and third part to ensure that there is a propositional symbol $P$ such that $\WW \nVdash P$ and therefore $\WW^\bot \nVdash P$.
\end{proof}

\begin{Pro}
	Let $\WW,v$ be any full model, and let $\alpha$ be a formula not containing $\bot$. Then
	\[ \WW \Vdash \alpha \miff \WW^\bot \Vdash \alpha\]
\end{Pro}
\begin{proof}
Induction on the complexity of formulas.
\end{proof}

We say that a partial order is \define{$\mathsf{v}$-free} if it doesn't contain $a,b,c$ such $a \leqslant b$, $a \leqslant c$, and $b$ and $c$  incomparable. 

\begin{Pro} \label{Pro:models_char_vfree}
Let $(\WW,\sqsubseteq, Q)$ be a structure. 
	\begin{enumerate}
  \item $\WW \Vdash \DGP$ if and only if $\WW$ is $\mathsf{v}$-free.
  \item   $\WW \Vdash \WLEM$ if  $\WW \setminus Q$ is $\mathsf{v}$-free.\footnote{The converse does not hold. See the model $\WW_4$ in Section \ref{sec:intuitionism}.}  
\end{enumerate}
\end{Pro}
\begin{proof}
\begin{enumerate}
  \item Assume $\WW$ is $\mathsf{v}$-free and let $P$ and $Q$ be arbitrary propositional symbols. Consider an arbitrary world $u \in \WW$. If  $P \in v(u)$ we have that $u \Vdash Q \implies P$. Similarly if $Q \in v(u)$ we have that $u \Vdash P \implies Q$. In both cases $u \Vdash P \implies Q \vee Q \implies P$. So assume that neither $P,Q \in v(u)$. Now consider the sets $A = \set{y \in \WW}{u \sqsubseteq y \wedge P \in v(y)} $ and $B = \set{y \in \WW}{u \sqsubseteq y \wedge Q \in v(y)} $. We must have that either $A \subset B$ or $B \subset A$: for assume there were $z,z^\prime$ such that $z \in A$, $z \notin B$, $z^\prime \in B$, and $z^\prime \notin A$. We must have that $z \not \sqsubseteq z^\prime$, since $z \sqsubseteq z^\prime$ implies that $z^\prime \in B$. Similarly $z^\prime \not \sqsubseteq z$, but then $u,z,z^\prime$ contradict the assumption that $\WW$ is $\mathsf{v}$-free. If  $A \subset B$ we have that for every $u \Vdash Q \implies P$, and if $B \subset A$ we get $u \Vdash P \implies Q$. In both cases $u \Vdash P \implies Q \vee Q \implies P$. Hence we have shown $\WW \Vdash \DGP$.

	Conversely we will show that if $\WW$ is not $\mathsf{v}$-free then $\WW \nVdash \DGP$. So assume there is $a,b,c \in \WW$ such that $a \sqsubseteq b$, $a \sqsubseteq c$, but neither $b \sqsubseteq c$ nor $c \sqsubseteq b$. Let $P_{b^\uparrow}$ and $P_{c^\uparrow}$ be the propositional symbols corresponding to the upwards closed sets $\{x \in\WW : b \sqsubseteq x\}$ and $\{x\in\WW :c \sqsubseteq x\}$, respectively. Notice that $b \Vdash P_{b^\uparrow}$, $b \nVdash P_{c^\uparrow}$, $c \Vdash P_{c^\uparrow}$, and $c \nVdash P_{b^\uparrow}$.
	Assume $a \Vdash (P_{b^\uparrow} \implies P_{c^\uparrow}) \vee (P_{c^\uparrow} \implies P_{b^\uparrow})$. Then either $a \Vdash (P_{b^\uparrow} \implies P_{c^\uparrow})$ or $a \Vdash (P_{c^\uparrow} \implies P_{b^\uparrow})$; w.l.o.g.\ the first case. Then by monotonicity $b \Vdash (P_{b^\uparrow} \implies P_{c^\uparrow})$. Since also $b \Vdash P_{b^\uparrow}$  we have $b \Vdash P_{c^\uparrow}$; a contradiction.
\item Assume $\WW \setminus Q$ is $\mathsf{v}$-free and let $P$ be an arbitrary propositional symbol. Consider an arbitrary world $u \in \WW$. If $u \in Q$ we have $u \Vdash \neg \alpha$ for any $\alpha$, so we may assume that $u \in \WW \setminus Q$. If there is no $u \sqsubseteq y$ such that $P \in v(y)$ and $y \notin Q$ then $u \Vdash \neg P$. So assume that there is $u \sqsubseteq y$ such that $y \notin Q$ and $P \in v(y)$. We want to show that $u \Vdash \neg \neg P$. To do this we will show that for any $u \sqsubseteq z$ if $z \Vdash \neg P$ then $z \in Q$. So assume that $u \sqsubseteq z$,  and $z \notin Q$. Because $\WW \setminus Q$ was assumed to be $\mathsf{v}$-free we either have $z \sqsubseteq y$ or $y \sqsubseteq z$. In the second case, by monotonicity, we get $z \Vdash P$ and therefore $z \Vdash \bot$; a contradiction. In the first case, similarly, by monotonicity $v \Vdash \neg P$ and therefore $y \Vdash \bot$; again a contradiction. So $z \in Q$. Hence we are done, since either $u \Vdash \neg P$ or $u \Vdash \neg \neg P$, so together $u \Vdash \neg P \vee \neg \neg P$. \qedhere
\end{enumerate}
\end{proof}

\begin{Pro}\label{pro:pw-incomp}
	Let $(\WW,\sqsubseteq, Q)$ be a structure. 
	\begin{enumerate}
	
  \item $\WW \Vdash \PP$ if and only if $\WW$ consists of pair-wise incomparable points.
  \item $\WW \Vdash \LEM$ if and only if $\WW \setminus Q$ consists of pair-wise incomparable points.
  \item $\WW \Vdash \DNE$ if and only if $\WW$ consists of pair-wise incomparable points and $Q = \emptyset$.

\end{enumerate}

\end{Pro}
\begin{proof}
\begin{enumerate}
  \item Assume that $\WW$ consists of pair-wise incomparable points, let $P,S$ be arbitrary propositional symbols, and consider an arbitrary world $u \in \WW$. If $S \in v(u)$ we have that $u \Vdash ((S \implies P) \implies S)\implies S$. If $S \notin v(u)$ we have that $u \Vdash (S \implies P)$, and therefore $u \nVdash ((S \implies P)  \implies S)$. Hence also in that case $u \Vdash ((S \implies P) \implies S)\implies S$.

Conversely assume that there are distinct $u,y \in \WW$ with $u \sqsubseteq y$. Now let $U = \set{u \sqsubseteq y}{u \neq y}$, and consider $P_U$ as above, and let $S = P_\emptyset$. Then for every $u \sqsubseteq y$ with $u \neq y$ we have $y \Vdash (P_U \implies S) \implies P_U$. For $u$ we have that $u \nVdash P_U \implies S$ (since $y \nVdash P_U \implies S $) and hence $u \Vdash (P_U \implies S) \implies P_U$. Now, if $u \Vdash \PP$ we would also have that $u \Vdash P_U$; a contradiction. So we have shown that if $\WW$ contains two comparable worlds then $\PP$ does not hold, or equivalently that if $\PP$ holds then $\WW$ does not contain comparable worlds.
\item Assume that $\WW \setminus Q$ consists of pair-wise incomparable points, let $P$ be an arbitrary propositional symbol, and consider an arbitrary world $u \in \WW$. Now either $P \in v(u)$ and therefore $u \Vdash P$ or $P \notin v(u)$. Since the only worlds above $u \sqsubseteq y$, $u \neq y$ are such that $y \in Q$ and therefore $y \Vdash \bot$ we have that for all $u \sqsubseteq y$ with  $y \Vdash P$ also $y \Vdash \bot$. Hence in that case $u \Vdash \neg P$. 

Conversely assume that there are distinct $u,y \in \WW \setminus Q$ with $u \sqsubseteq y$. Consider $P_{y^\uparrow}$ as above. Then $u \nVdash P_{y^\uparrow}$. But also $u \nVdash \neg P_{y^\uparrow}$, since if $u \Vdash \neg P_{y^\uparrow}$ also $y \Vdash \neg P_{y^\uparrow}$ which would imply $y \Vdash \bot$; a contradiction to $y \notin Q$. Hence $u \nVdash \LEM$. Equivalently we have shown that if $\WW \Vdash \LEM$ then $\WW \setminus Q$ cannot contain two comparable worlds.
\item This follows from Proposition \ref{Pro:EFGiffQempty} together with either of the previous two items and Proposition \ref{Pro:LEM_and_EFQ_iff_DNE} further down. \qedhere
\end{enumerate}
\end{proof}

To start with  one of the most trivial models imaginable, consider $\WW_1=\menge{0}$ and $v_1(0) = \emptyset$. That is there is only one world. By Proposition \ref{Pro:loBOTomy} this means that $$ \LEM \not\mimplies \DNE $$
and $$ \LEM \not\mimplies \EFQ \ . $$

The second structure we are considering is $\WW_2 = (\menge{1,2},\leqslant,\emptyset)$. 
Since $Q = \emptyset$ this is a model of \EFQ. It does not, however, force \DNE, since, if $P$ is such that $\ext{P} = \menge{2}$ then no node forces $\neg P$, whence $\WW_2 \Vdash \neg \neg P$. However, of course, $1 \nVdash P$. The same $P$ also ensures that $\WW_2 \nVdash \LEM$. A useful variant of $\WW_2$ is $\WW_2^\prime = (\menge{1,2},\leqslant,\{2\})$; in that case, with $\ext{P}=\menge{2}$ again, $\DGPimp$ is forced since it is $v$-free; however, $\WT$ is not forced. To see this, note that since $\bot$ is forced whenever $P$ is forced, $1 \Vdash \neg P$. However, since $P\implies S$ is nowhere forced, $1 \Vdash \neg(P\implies S)$. Since $1\nVdash \bot$, $1 \nVdash \neg\neg(P\implies S)$. Hence $\WW_3^\prime\nVdash \WT$. Note that $\LEM,\DGP$, and $\WLEM$ are also forced in this model, separating $\WT$ and $\DGPimp$ from these principles.

The third structure $(\WW_3,\emptyset)$ is a simple $v$ shape, so, for example, $(\{ \emptyset, \{1\},\{2\} \}, \subset, \emptyset)$. $\WW_3$ does not satisfy \WLEM, since for $P = P_{\{1\}}$ we have that $\emptyset \nVdash \neg P \vee \neg \neg P$. Like $\WW_2$, a useful variant is $\WW_3^\prime = (\{ \emptyset, \{1\},\{2\} \}, \subset, \menge{\{1\},\{2\}} )$. This model does not force \DGPimp. To see this, assign $P=P_{\{1\}},S=S_{\{2\}}$. Then since $\emptyset \nVdash P\implies S$ and $\emptyset \nVdash S \implies P$, and since both $\{1\}\Vdash \bot$ and $\{2\}\Vdash \bot$, $\emptyset \Vdash \neg(S\implies P)$ and $\emptyset \Vdash \neg(P\implies S)$. But since $\emptyset \nVdash\bot$, we have $\emptyset\nVdash\neg\neg(P \implies S)$. Hence $\WW_3^\prime \nVdash \DGPimp$.

Figure \ref{fig:hasse}  includes the Hasse diagrams of the underlying orders of these models.

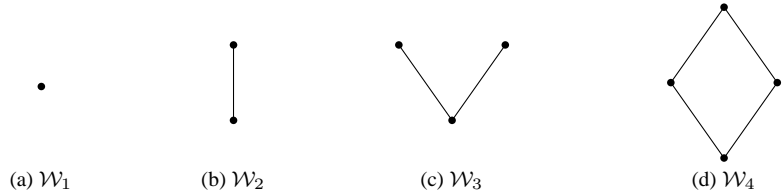
\begin{figure}[ht]
\caption{The Hasse diagrams of the underlying orders of our models
      \label{fig:hasse}}
\floatbox{figure}{%
      }%
    {
\begin{subfloatrow}

\subfloat[$\WW_1$]{

\begin{tikzpicture}
\useasboundingbox (-1,0) rectangle (1,0.1);
  \node [circle,fill=black,inner sep=1pt] at (0,0) {};
\end{tikzpicture}
}

\subfloat[$\WW_2$]{
\begin{tikzpicture}
\useasboundingbox (-1,0) rectangle (1,1);
  \node [circle,fill=black,inner sep=1pt] at (0,0) {};
  \node [circle,fill=black,inner sep=1pt] at (0,1) {};
  \draw (0,0) -- (0,1);
\end{tikzpicture}
}
\subfloat[$\WW_3$]{
\begin{tikzpicture}
\useasboundingbox (-1.7,0) rectangle (1.7,1);
  \node [circle,fill=black,inner sep=1pt] at (0,0) {};
  \node [circle,fill=black,inner sep=1pt] at (-0.7,1) {};
  \node [circle,fill=black,inner sep=1pt] at (0.7,1) {};
  \draw (0,0) -- (-0.7,1);
  \draw (0,0) -- (0.7,1);
\end{tikzpicture}
}
\subfloat[$\WW_4$]{
\begin{tikzpicture}
\useasboundingbox (-1.7,0) rectangle (1.7,2);
  \node [circle,fill=black,inner sep=1pt] at (0,0) {};
  \node [circle,fill=black,inner sep=1pt] at (-0.7,1) {};
  \node [circle,fill=black,inner sep=1pt] at (0.7,1) {};
  \node [circle,fill=black,inner sep=1pt] at (0,2) {};
  \draw (0,0) -- (-0.7,1);
  \draw (0,0) -- (0.7,1);
  \draw (0,2) -- (-0.7,1);
  \draw (0,2) -- (0.7,1);
\end{tikzpicture}
}
\end{subfloatrow}
}\end{figure}

These models, together with their $^\bot$ versions and useful variants (and $\WW_4$; see the next section), show that all the implications in Figure \ref{fig:1} are strict. These models also show that---apart from the transitive closure---no arrows can be added to Figure \ref{fig:1}.  For simplicity we include a table.
\setcounter{figure}{3}
\begin{figure}[ht] \centering 
\begin{tabular}{c|c|c|c|c|c|c|c|c}
& \DNE & \EFQ & \LEM & \DGP & \PP & \WLEM & \WT & \DGPimp \\	\hline
$\WW_1$ & \cmark  & \cmark & \cmark & \cmark & \cmark & \cmark & \cmark  & \cmark \\
$\WW^\bot_1$ & \xmark  & \xmark & \cmark & \cmark & \cmark & \cmark & \cmark  & \cmark \\
$\WW_2$ & \xmark  & \cmark & \xmark & \cmark & \xmark & \cmark & \cmark  & \cmark \\
$\WW_2^\bot$ & \xmark  & \xmark & \cmark & \cmark & \xmark & \cmark & \cmark  & \cmark \\
$\WW_2^\prime$ & \xmark & \xmark & \cmark & \cmark & \xmark & \cmark & \xmark  & \cmark \\
$\WW_3$ & \xmark  & \cmark & \xmark & \xmark & \xmark & \xmark & \cmark & \cmark \\
$\WW_3^\bot$ & \xmark  & \xmark & \cmark & \xmark & \xmark & \cmark & \cmark  & \cmark \\
$\WW_3^\prime$ & \xmark & \xmark & \cmark & \xmark & \xmark & \cmark & \xmark  & \xmark \\
$\WW_4$ & \xmark  & \cmark & \xmark & \xmark & \xmark & \cmark & \cmark & \cmark \\
$\WW_4^\bot$ & \xmark  & \xmark & \cmark & \xmark & \xmark & \cmark  & \cmark  & \cmark
\end{tabular}
\caption{An overview, over all models over minimal logic. See next section for $\WW_4$.}
\end{figure}
Further, note that since every principle fails to be enforced in at least one of the models, none of the principles outlined are provable in minimal logic alone.

\section{The intuitionistic case}\label{sec:intuitionism}

It is well known that:
\begin{Pro} \label{Pro:LEM_and_EFQ_iff_DNE}
	\LEM and \EFQ imply \DNE
\end{Pro}
\begin{proof}
Let $\varphi$ be such that $\neg \neg \varphi$ holds. Thus if we have $\neg \varphi$ then $\bot$, which by \EFQ implies $\varphi$. Together with \LEM that means that $\varphi \vee \neg \varphi$ implies $\varphi$.	
\end{proof}
Thus under the assumption of \EFQ our hierarchy collapses to the following simple one:
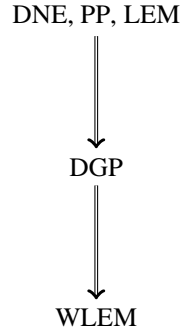
\begin{figure}[ht] \centering
\begin{tikzpicture}[node distance=2 cm, auto]
  \node (DNE2) {\DNE, \PP, \LEM};
  \node (DGP2) [below of=DNE2] {\DGP};
  \node (WLEM2) [below  of=DGP2] {\WLEM};
  \draw[double,->] (DNE2) to node {} (DGP2);
  \draw[double,->] (DGP2) to node {} (WLEM2);
\end{tikzpicture}
\end{figure}

The model $\WW_2$ shows that the first of these implications is strict, however none of the models considered so far satisfy \EFQ and \WLEM, but not \DGP. To do so we consider the  diamond-shaped structure $(\{ \emptyset, \{1\}, \{2\}, \{1,2\} \},\subset, \emptyset)$. As $Q = \emptyset$ this is a model of \EFQ. As $\WW_4$ is not $\textsf{v}$-free it does not satisfy \DGP (Proposition \ref{Pro:models_char_vfree}).

Finally we need to check that \WLEM holds. So let $P$ be an arbitrary propositional symbol. If $P \notin v(\{1,2\})$ then $\WW_4 \Vdash P \implies \bot$, since $P$ is never forced. If $P \in v(\{1,2\})$ then $P \implies \bot$ is never forced in any world, so $\WW_4 \Vdash \neg \neg P$ vacuously. In both cases $\WW_4 \Vdash \neg P \vee \neg \neg P$.

\section{Open Questions and Concluding Remarks}

\subsection{Statement \texorpdfstring{$\ref{PMI7}^\implies$}{7->}}

As shown in Section \ref{Sec:impl_fragment} we have 
\[ \WT \mimplies \ref{PMI7}^\implies  \mimplies \DGPimp \ .  \]
One can check that $\ref{PMI7}^\implies$ holds in exactly the models discussed which validate \DGPimp, so it is provably weaker than \WT and one would naturally conjecture that it is equivalent to that statement. However we could not find a proof of $\DGPimp \mimplies \ref{PMI7}^\implies$. It is also interesting that we could not find a proof of the converse that does not rely on $\wedge$.

\subsection{Kreisel Putnam and Scott logic}
In the  fabulously  titled (even for German speakers) paper \cite{gKhP57} it is shown that the formula
\begin{equation} \label{KP} \tag{\KP} 
(\neg \varphi \implies \psi \vee \vartheta) \implies (\neg \varphi \implies \psi) \vee (\neg \varphi \implies \vartheta) 
\end{equation}

is not derivable in intuitionistic logic. One can check that it holds in exactly the same of the models discussed as \DGPimp. It is completely unclear, though, whether either one implies the other. In the same paper \cite{gKhP57} also the following formula, which is not derivable intuitionistically, and which is due to Dana Scott is mentioned 
\[ ((\neg \neg \varphi \implies \varphi) \implies (\varphi \vee \neg \varphi)) \implies (\neg \neg \varphi \vee \neg \varphi).  \]
This formula is clearly implied by \WLEM. It is, however, weaker, since it actually holds in all models we have considered in this paper. It would also be interesting to have intuitionistic models (i.e.\ ones validating \EFQ) rejecting these principles.

\subsection{SmL}
It is also worth mentioning another important formula characterising an superintuitionistic logic, the so called Smetanich's logic \cite{fW07}. It is obtained by adding the following axiom scheme to intuitionistic logic.
\begin{equation} \tag{\SmL} \label{SmL} 
 (\neg \psi \implies \varphi) \implies (((\varphi \implies \psi) \implies \varphi) \implies \varphi) \ . 
 \end{equation}
It is easily seen that this is implied by \LEM. It also implies \WLEM which we can see if we apply it to $\neg \varphi \vee \neg \neg \varphi$ and $\neg \varphi$.
However we can check that it holds in $\WW_2$ and does not hold in $\WW_3^\prime$ which means that it neither implies \LEM nor is it implied by or implies \DGP (and therefore neither by \WLEM). That means we have the following extension to the left bottom corner of Figure \ref{fig:1}.

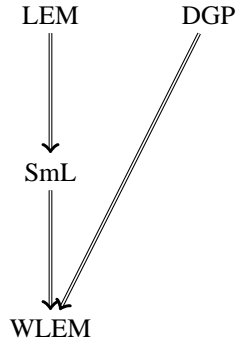
\begin{figure}[ht]\centering
\begin{tikzpicture}[node distance=.1\textheight, auto]
  \node (LEM) {\LEM};

  \node (DGP) [right of=LEM] {\DGP};
  \node (SmL) [below of=LEM] {\SmL};

  \node (WLEM) [below  of=SmL] {\WLEM};

  \draw[double,->] (LEM) to node {} (SmL);
  \draw[double,->] (SmL) to node {} (WLEM);
  \draw[double,->] (DGP) to node {} (WLEM);
\end{tikzpicture}
\caption{A small extension in the lower left corner of Figure \ref{fig:1}.}
\end{figure}

It is not clear, however, how \SmL relates to \DGP intuitionistically. One can see that \DGP does not imply \SmL even under the assumption of \EFQ, since the former holds in the model $(\WW_5,\emptyset, \leqslant)$, where $\WW_5 = (\{1,2,3 \})$ and $\leqslant$ is the usual order.

\subsection{Concluding Remarks}

The distinctions and equivalences in this paper represent, of course, only the tip of the proverbial iceberg. Beyond what can be distinguished in other non-classical logics, of special mention are substructural logics \cite{fP13}. The proofs presented here assume all the usual structural rules. An analysis of which proofs still go through in the various substructural logics will clearly shed further light on computational aspects.

\end{document}